\documentclass[12pt,oneside]{amsart}

\usepackage{amsfonts}
\usepackage{amsmath, amsthm, amssymb, ulem, amscd}

\def\crn#1#2{{\vcenter{\vbox{
        \hbox{\kern#2pt \vrule width.#2pt height#1pt
           }
          \hrule height.#2pt}}}}
\def\intprod{\mathchoice\crn54\crn54\crn{3.75}3\crn{2.5}2}
\def\into{\mathbin{\intprod}}

\newcommand{\stopthm}{\hfill$\square$\medskip}

\newcommand{\pa}{\partial}

\newcommand{\Eval}{\operatorname{Eval}}

\renewcommand{\span}{\operatorname{span}}

\renewcommand{\Re}{\operatorname{Re}}

\newcommand{\bx}{\mathbf x}
\newcommand{\bp}{\mathbf p}

\newcommand{\cU}{\mathcal{U}}

\newcommand{\cJ}{\mathcal{J}}
\newcommand{\cI}{\mathcal{I}}
\newcommand{\cS}{\mathcal{S}}
\newcommand{\cF}{\mathcal{F}}
\newcommand{\R}{\mathbb R}

\newcommand{\ep}{\epsilon}
\newcommand{\om}{\omega}
\newcommand{\ga}{\gamma}
\newcommand{\al}{\alpha}

\newcommand{\xh}{\widehat{x}}

\newcommand{\cL}{\mathcal{L}}

\newcommand{\be}{\beta} 
\newcommand{\gb}{\bar{g}}

\theoremstyle{plain}
\newtheorem{theorem}{Theorem}[section]
\newtheorem{lemma}[theorem]{Lemma}
\newtheorem{proposition}[theorem]{Proposition}

\theoremstyle{definition}

\newtheorem{definition}[theorem]{Definition}

\theoremstyle{remark}
\newtheorem{remark}[theorem]{Remark}

\numberwithin{equation}{section}

\title[Normal Form for Edge Metrics]{Normal Form for Edge Metrics}     

\author{C. Robin Graham}
\address{Department of Mathematics, University of Washington,
Box 354350\\
Seattle, WA 98195-4350}
\email{robin@math.washington.edu}

\author{Joshua M. Kantor}
\address{MIT Lincoln Laboratory, Lexington, MA 02420-9108}   
\email{joshua.kantor@ll.mit.edu}

\begin{document}

\maketitle

\thispagestyle{empty}

\renewcommand{\thefootnote}{}
\footnotetext{Partially supported by NSF grant \# DMS 0906035.}   
\renewcommand{\thefootnote}{1}

\section{Introduction}\label{intro}

An edge metric is a metric on the interior of a
manifold-with-boundary which is singular at the boundary in a manner
described by a given fibration of the boundary.  The related 
edge differential operators arise in many settings and have been the
subject of much research; see, for example, \cite{Ma}.  Our interest in edge
metrics arises from the 
observation that they are a robust class of metrics naturally generalizing
the  
product of a conformally compact metric with a metric on a compact
manifold.  The AdS/CFT correspondence in physics deals with such
product metrics, and the arena of edge metrics appears 
to be a natural setting for the geometric and analytic 
questions which arise.  The thesis \cite{Ka} considers a problem
concerning eleven-dimensional supergravity from this point of view.  
In this paper we derive a normal form for edge metrics which we expect will
be useful in further studies.  The normal form is the analogue of geodesic
normal coordinates relative to the boundary at infinity.  

Edge metrics reduce to conformally compact metrics in the special case that
the fibers of the boundary are points.  Asymptotically hyperbolic (AH)
metrics are conformally compact metrics satisfying a particular scalar 
normalization at infinity.  The normal form for AH metrics was derived in  
\cite{GL} and a different proof was given in \cite{JS2}.  This normal form  
has been useful in a number of problems concerning AH metrics.  The
existence statement is that if $g$ is
AH on $X$, then there is a diffeomorphism $\psi$ from a neighborhood 
of $\{0\}\times \pa X$ in $[0,\infty)\times \pa X$ to a neighborhood  
of $\pa X$ in $X$ such that $\psi|_{\pa X}=Id$ and 
\begin{equation}\label{nform}
\psi^*g=\frac{dx^2+h_x}{x^2},
\end{equation}
where $x$ is the coordinate in $[0,\infty)$ and $h_x$ is a 1-parameter
family of metrics on $\pa X$.  The normal form for $g$ is not unique and  
this is a crucial point.  
There is a conformal class of metrics on $\pa X$, called the conformal
infinity of $g$, and the normal forms for $g$ are parametrized precisely by
the representative metrics in the conformal infinity:  $h_0$ is the given
conformal representative.  Conformal rescalings  
on the boundary thus correspond via the normal form to diffeomorphism
changes on the interior.  One application of the AH normal form is to the
renormalized volume of an AH Einstein metric (see \cite{G}).  The
renormalized volume is defined in terms of 
the exhaustion determined by a defining function $x$ 
in a normal form, and its invariance or noninvariance 
(measured by a conformal anomaly) is determined via properties of the  
diffeomorphism determined by a conformal change on the boundary.  In
\cite{JS2} the normal form  
\eqref{nform} arose in an inverse scattering context, where it was used  
to normalize the action of diffeomorphisms on AH metrics.   

For edge metrics there is an analogue of the AH scalar normalization, 
which defines 
what we call a normalized edge metric.  There is another necessary
condition for an edge metric to have a normal form which is vacuous in the
conformally compact case.  It is that a 1-form on the boundary fibers
constructed out of the metric and a defining function be
globally smoothly exact; this condition is independent of the choice of
defining function.  We call an edge metric satisfying this condition
an exact edge metric.  An exact edge metric has a reduced conformal
infinity analogous to the conformal infinity in the conformally compact
case.  In the edge case this is a conformal class of metrics on 
$TX|_{\pa X}/V$, where  
$V$ is the vertical bundle of the boundary fibration (i.e. the tangent
bundle to the fibers), and any two metrics in the conformal class differ by
a positive function which is locally constant on the boundary fibers.  Our
main result, Theorem~\ref{main}, asserts that an exact, normalized edge 
metric $g$ can be  
put into normal form near the boundary by a diffeomorphism which restricts
to the identity on the boundary, and the different normal forms for $g$ are   
parametrized precisely by the representatives for the reduced conformal
infinity.  Thus the result takes the same form as the result for AH
metrics.   

The proof of the normal form for AH metrics in \cite{GL} is based on the
observation that if $\psi$ puts $g$ into normal form as above, then 
$\xh:=x\circ\psi^{-1}$ satisfies the eikonal equation
\begin{equation}\label{eikonal1}
\Big|\frac{d\xh}{\xh}\Big|^2_g=1.
\end{equation}
By first solving this equation on $X$, one can 
therefore obtain directly the $x$-component of $\psi^{-1}$.  And once one
has $\xh$, it is clear from the normal form \eqref{nform} that the full map
$\psi$ can be 
constructed by following the integral curves of the noncharacteristic 
vector field $\xh^{-1}X_{\xh}$, where $X_{\xh}$ is the vector field dual to
$\frac{d\xh}{\xh}$ with respect to $g$.  Now \eqref{eikonal1} appears to be
a singular 
equation.  But by writing $\xh=e^\om x_0$, where $x_0$ is a fixed defining
function and $\om$ a new unknown, then expanding the left-hand side of
\eqref{eikonal1} and moving a term to the right-hand side, and finally 
dividing by $x_0$, it becomes a noncharacteristic
first order nonlinear pde for $\om$, which is solvable by the method of 
characteristics.  This method of reducing a singular eikonal equation to 
a noncharacteristic initial value problem at infinity and then constructing 
the diffeomorphism by following integral curves of an associated
noncharacteristic vector field can be used to derive normal forms in a
number of other settings.  It  
gives a simple alternate derivation of the normal form for scattering metrics
proved in \cite{JS1}, and is used in \cite{GS} to derive a 
normal form for $\Theta$-metrics.  (In the case of scattering metrics, the 
corresponding eikonal equation is $\big|\frac{d\xh}{\xh^2}\big|^2_g=1$.  One
writes $\xh=x_0+\omega x_0^2$ for a new unknown $\omega$ and divides the
equation by $x_0^2$ rather than $x_0$ to obtain a noncharacteristic initial
value problem for $\omega$.  The diffeomorphism is constructed by
following the integral curves of $\xh^{-2}X_{\xh}$, where $X_{\xh}$ is the
vector field dual to $\frac{d\xh}{\xh^2}$ with respect to $g$.)   

In the edge case we were unable to reduce \eqref{eikonal1} to a
noncharacteristic problem.  The issue is the following.  Equation
\eqref{eikonal1} involves the components of the inverse metric $g^{-1}$.
In the cases of conformally compact or scattering or $\Theta$-metrics, all
components of 
$g^{-1}$ vanish at the boundary.  Because of this, one can divide
\eqref{eikonal1} by the correct power of $x_0$ as described above to obtain
a noncharacteristic problem, and still have an equation with smooth 
coefficients up to the boundary.  
But for edge metrics, the components of $g^{-1}$ along the fibers do not
vanish at $\pa X$.  The coefficient of the derivative transverse to the
boundary still vanishes there, and one cannot carry out  
the division by $x_0$ to make the problem noncharacteristic.  So it seems
that one is forced to deal with a singular equation.    

Joshi studied the normal form for $b$-metrics in \cite{J}.  
A $b$-metric is an edge metric in the opposite extreme case from a
conformally compact metric:  the case in which there is only one boundary 
fiber, the boundary itself.  Although geometrically simpler than
general edge metrics, $b$-metrics already exhibit the essential analytic
difficulty 
of the general case as far as derivation of the normal form.  By formal
calculations, Joshi 
derived the normal form for exact, normalized $b$-metrics modulo error
terms vanishing 
to infinite order at the boundary using the method of \cite{JS1},
\cite{JS2}, 
but was unable to obtain the result in an open set.  In this regard he
commented, ``it is not clear how to proceed''.  The normal form modulo 
infinite-order vanishing error terms was derived for general exact, 
normalized edge metrics in \cite{Ka} by formal analysis of
\eqref{eikonal1}.  (We remark that the literature concerning  
$b$-metrics, for instance \cite{Me} and \cite{J}, defines an exact 
$b$-metric to be    
what we refer to as an exact, normalized $b$-metric.  That is, we 
separate the two separately invariant conditions: exactness and
normalization.)  

To solve \eqref{eikonal1} in the edge case, we derive a result concerning
existence and uniqueness of certain characteristic nonlinear first-order
initial value problems for a real scalar unknown.  We consider problems of
the form  
\begin{equation}\label{ivp}
x\pa_x\om = F(x,y,\om,\pa_y\om),\qquad \om(0,y)=\om_0(y).
\end{equation}
Here $(x,y)$ are coordinates on $\R^n$, $n\geq 1$, with $x\geq 0$,
$y\in\R^{n-1}$, and   
the unknown function $\om(x,y)$ is real-valued.  $F(x,y,\om,q)$ is a smooth
real function of $(x,y,\om,q)\in \R\times \R^{n-1}\times \R\times  
\R^{n-1}$ and the initial value $\om_0$ is a smooth function of $y$.   
\begin{theorem}\label{pdetheorem}
Suppose that for all $y$ one has 
\begin{equation}\label{first}
F(0,y,\om_0(y),\pa_y\om_0(y))=0,\qquad F_q(0,y,\om_0(y),\pa_y\om_0(y))=0
\end{equation}
and 
\begin{equation}\label{second}
F_\om(0,y,\om_0(y),\pa_y\om_0(y))<1.
\end{equation}
Then there exists a unique smooth solution of \eqref{ivp} 
for sufficiently small $x\geq 0$.  
\end{theorem}

We prove Theorem~\ref{pdetheorem} by an adaptation of the method of 
characteristics.  In 
the noncharacteristic case this proceeds by solving the 
ordinary differential equations corresponding to a Hamiltonian flow-out in
the first jet bundle of the solution.  Because our initial value problem is
characteristic, the Hamiltonian vector field vanishes identically on the  
initial submanifold, so there is no flow-out in the usual sense.
Nonetheless, we are able to construct a flow-out by considering 
what we call characteristic integral curves emanating from a zero of a
vector field.  We show these exist under an appropriate hypothesis on the
eigenvalues of the linearization of the vector field at the zero 
(Theorem~\ref{vectfield}) and the union of the characteristic integral
curves of the Hamiltonian vector field starting from the initial
submanifold gives the 1-jet of the 
solution.  It seems likely that Theorem~\ref{pdetheorem} and the method of 
using these characteristic integral curves will be useful in other
problems.  Once we have solved \eqref{eikonal1}, the second part of the
proof of the normal form, constructing the diffeomorphism by flowing along
integral curves, works just as before:  the relevant vector field
$\xh^{-1}X_{\xh}$ is smooth and noncharacteristic.  

In \S 2 we define edge metrics and study the geometric
structure they induce at infinity.  We work with edge metrics of arbitrary 
signature, under an additional nondegeneracy hypothesis which we call
horizontal nondegeneracy.  This is automatic in the case of definite
signature.  We define the notions of normalization, exactness, and
reduced conformal infinity referred 
to above, and show that an arbitrary horizontally nondegenerate edge metric
invariantly induces a metric on each fiber of the boundary.  We formulate
the normal form condition and state the main result, Theorem~\ref{main},
which asserts the existence and uniqueness of the normal form.  We show
that Theorem~\ref{main} follows from the solvability of the eikonal
equation, and reduce the eikonal equation to a problem of the form
\eqref{ivp}.  In \S 3 we prove Theorem~\ref{pdetheorem}.  

Throughout, by smooth we mean infinitely differentiable, and all objects
are assumed smooth unless explicitly stated otherwise.

\section{Edge Metrics}\label{geometry}

Let $(X,\pa X)$ be a manifold-with-boundary.  Suppose that $\pa X$ is the
total space of a fibration
$$
\begin{CD}
F  @>>>  \pa X\\
@.          @VV{\pi}V\\
   @.    Y
\end{CD}
$$
with fiber $F$ and base $Y$.  One says that $(X,\pa X)$ is an edge,
or boundary-fibered, manifold.  A motivating special case is that of a
product $X=M\times F$, where $(M,\pa M)$ is a manifold-with-boundary.
Then $\pa X = \pa M\times F$ with projection $\pi:\pa X\rightarrow \pa M 
=Y$.  

Let $V=\ker \pi_*\subset T\pa X$ denote the vertical vector bundle over
$\pa X$.   
A vector field $\xi$ on $X$ is said to be an {\it edge vector field} if its 
restriction to $\pa X$ is tangent to the fibers.  Equivalently, one
requires that $\xi|_{\pa X}\in \Gamma(V)$.    

Near a point of $\pa X$ one can choose local coordinates $(x,y^\alpha,z^A)$
so that $x$ is a defining function for $\pa X$ with $x>0$ in $X^\circ$, 
$y^\al$ are coordinates on $X$ whose restrictions to $\pa X$ are lifts
of local coordinates on $Y$ (so the $y^\al$ are constant on the fibers),
and $z^A$ are coordinates on $X$ whose 
restrictions to each fiber of $\pa X$ are local coordinates on the fiber.  
The edge vector fields  are then
$\span_{C^\infty(X)}\{x\pa_x,x\pa_{y^\al},\pa_{z^A}\}$.  

There is a vector bundle ${}^eTX$ on $X$, the {\it edge tangent bundle},
which 
can be characterized by the requirement that its smooth sections are the 
edge vector fields.  $x\pa_x$ and the $x\pa_{y^\al}$ define sections of
${}^eTX$ which are nonvanishing on $\pa X$, and in local coordinates near 
$\pa X$ one has 
${}^eT_pX=\span\{x\pa_x,x\pa_{y^\al},\pa_{z^A}\}$ for $p\in X$.  The dual
edge cotangent bundle ${}^eT^*X$ has fibers spanned by the dual basis
$\frac{dx}{x}$, $\frac{dy^{\al}}{x}$, $dz^A$.  
There is
a well-defined evaluation map $\Eval:{}^eTX\rightarrow TX$ with the
property 
that $\Eval_p:{}^eT_pX\rightarrow T_pX$ is an isomorphism for $p\in
X\setminus \pa X$, but $\Eval_p({}^eT_pX)=V_p$ for $p\in \pa X$.  For 
$p\in \pa X$, we define
the horizontal bundle $H_p=\ker \Eval_p\subset {}^eT_pX$, so that 
$H_p=\span\{x\pa_x,x\pa_{y^\al}\}$.  

An {\it edge metric} $g$ on an edge manifold $X$ is a smooth nondegenerate 
section of 
$S^2({}^eT^*X)$.  (We allow metrics of arbitrary signature.)  In local
coordinates near $\pa X$, if we write  
\begin{equation}\label{gform}
g=\left(\begin{array}{ccc}
\frac{dx}{x}&
\frac{dy^{\bf \al}}{x}&
dz^A\\
\end{array}\right)
\left(\begin{array}{ccc}
\gb_{00}&\gb_{0\be}&\gb_{0B}\\
\gb_{\al 0}&\gb_{\al\be}&\gb_{\al B}\\
\gb_{A0}&\gb_{A\be}&\gb_{AB}
\end{array}\right)\left(\begin{array}{c}
\frac{dx}{x}\\
\frac{dy^{\be}}{x}\\
dz^B\\
\end{array}\right),
\end{equation}
then this is the requirement that the $\gb$-matrix be smooth and 
nondegenerate up to $\pa X$.  An edge metric restricts to a usual 
metric on $X^\circ$ which is singular at $\pa X$ with the form of the 
singularity determined by the boundary fibration.  In the special case that 
the fiber $F$ is a point, there are no $z$ variables and an edge metric is
just a conformally compact metric on the manifold-with-boundary 
$(X,\pa X)$.  In the 
special case that $Y$ is a point, there are no $y$ variables, and in this
case edge metrics are called $b$-metrics on the
manifold-with-boundary $(X,\pa X)$.  For a product edge manifold 
$X=M\times F$ as above, a product edge metric is a metric of the form  
$g=g_M+g_F$, where $g_M$ is a conformally compact metric on $(M,\pa M)$ and
$g_F$ is a metric on $F$.  

We will say that an edge metric is {\it horizontally nondegenerate} if 
$g|_{H_p}$ is nondegenerate for all $p\in \pa X$.  Clearly any positive
definite edge metric is horizontally nondegenerate.  A horizontally
nondegenerate metric 
induces a metric on the bundle $H\subset {}^eTX|_{\pa X}$.  This induced
metric may be reinterpreted as a conformal class of metrics on 
$TX|_{\pa X}/V$ as follows.  Multiplication by a defining function $x$ 
induces a map $T_pX\rightarrow {}^eT_pX$ for $p\in \pa X$ with kernel $V_p$
and range $H_p$.  This induces an isomorphism $H\cong TX|_{\pa X}/V$
dependent on the choice of $x$ only up to scale.  Via this isomorphism we
can transfer the metric $g|_H$ to a conformal class of metrics on 
$TX|_{\pa X}/V$, called the conformal infinity of $g$.  (This choice of
terminology is slightly at odds with usual usage in the 
conformally compact case, where the conformal infinity is typically regarded
as a metric on $T\pa X$ rather than on $TX|_{\pa X}$.  One reason for our 
choice is that $T\pa X/V$ has rank 0 for the special case of $b$-metrics.)      
Of course one can realize the representatives in the conformal infinity
directly without recourse to $H$:  the metric $(x^2g)|_{X^\circ}$ extends  
smoothly to $X$ as a section of $S^2T^*X$ and $(x^2g)|_{\pa X}$
annihilates $V$, so it induces a quadratic form on the bundle 
$TX|_{\pa X}/V$ over $\pa X$ which is nondegenerate iff $g$ is horizontally 
nondegenerate.

There is a generalization to edge metrics of the normalization condition
that a conformally compact metric be asymptotically hyperbolic.  The
section $\frac{dx}{x}$ of ${}^eT^*X$ restricts to a section of 
$H^*$ which is independent of $x$.  This is because if $\widehat{x}=ax$
with $0<a\in C^\infty(X)$, then   
$\frac{d\widehat{x}}{\widehat{x}}=\frac{dx}{x}+\frac{da}{a}$, and 
the restriction of $\frac{da}{a}$ to $H$ vanishes.  If $g$ is horizontally
nondegenerate, then we can consider the length squared of $\frac{dx}{x}|_H$
with respect to $g|_H$, and this will be an an invariant of $g$.   
\begin{definition}\label{normalized}
An edge metric $g$ is said to be {\it normalized} if $g$ is
horizontally nondegenerate and if 
$$
\bigg|\frac{dx}{x}\big|_H\bigg|^2_{g|_H}=1\qquad \mbox{on}\;\; \pa X.   
$$ 
\end{definition}

\begin{remark}
 In the case of indefinite signature, one could equally well consider the 
condition $\big|\frac{dx}{x}\big|_H\big|^2_{g|_H}=-1$.  Our treatment
applies to this case upon replacing $g$ by $-g$.  
\end{remark}

We can also consider the length squared of $\frac{dx}{x}$ with respect to
$g$ on all of ${}^eTX$.  In general this will depend on $x$ since
$\frac{da}{a}|_{\pa X}$ is a nontrivial section of ${}^eT^*X|_{\pa X}$ if
$a$ varies along the fibers.   

\begin{definition}\label{gnormalized}
Let $g$ be an edge metric and $x$ a defining function for $\pa X$.  
$x$ is said to be {\it $g$-normalized} if $|\frac{dx}{x}|^2_g=1$ on $\pa
X$.    
\end{definition}

We make these invariant conditions explicit in local coordinates.  If
$g$ is written as \eqref{gform}, then
$g|_H$ is represented by 
$
\left(\begin{array}{cc}
\gb_{00}&\gb_{0\be}\\
\gb_{\al 0}&\gb_{\al\be}\\
\end{array}\right).
$
Horizontal nondegeneracy of $g$ is the requirement that this quadratic
form be nondegenerate at $\pa X$.  In this case, the dual metric is given
by the inverse matrix, so we write 
$$
\left(\begin{array}{cc}
\gb_{00}&\gb_{0\be}\\
\gb_{\al 0}&\gb_{\al\be}\\
\end{array}\right)^{-1}
=\left(\begin{array}{cc}
C&*\\
*&*\\
\end{array}\right),
$$
and $C=\big|\frac{dx}{x}\big|_H\big|^2_{g|_H}$.  
Thus $g$ is normalized means exactly that $C=1$,
and this condition is independent of the choice of all the coordinates.   

On the other hand, $\frac{dx}{x}$ is a dual basis vector in the full frame.   
We write 
$$
\left(\begin{array}{ccc}
\gb_{00}&\gb_{0\be}&\gb_{0B}\\
\gb_{\al 0}&\gb_{\al\be}&\gb_{\al B}\\
\gb_{A0}&\gb_{A\be}&\gb_{AB}
\end{array}\right)^{-1}
=
\left(\begin{array}{ccc}
\gb^{00}&\gb^{0\be}&\gb^{0B}\\
\gb^{\al 0}&\gb^{\al\be}&\gb^{\al B}\\
\gb^{A0}&\gb^{A\be}&\gb^{AB}
\end{array}\right),
$$
and then $\gb^{00}=|\frac{dx}{x}|_g^2$.  So $x$ is $g$-normalized means  
$\gb^{00}=1$.  This
condition is independent of the choice of $y^\al$, $z^A$, but in general
does depend on the choice of $x$.  

A horizontally nondegenerate edge metric invariantly induces a 
pseudo-Riemannian metric on the fibers of $\pa X$.  Let $g$ be an edge
metric.  The induced 
dual metric on ${}^eT^*X$ is a section $g^{-1}$ of $S^2({}^eTX)$.  Thus 
$\Eval(g^{-1})$ is a smooth section of $S^2TX$.  On $\pa X$, this section
of $S^2TX$ degenerates:  its restriction to $\pa X$ is a smooth section of
$S^2V\subset S^2TX$.  Elementary linear algebra (most easily carried out in
terms of the explicit
formulation of these conditions below) shows that the condition that 
$g$ is horizontally nondegenerate is equivalent to the 
condition that $\Eval(g^{-1})$ is a nondegenerate section of $S^2V$.  So if
$g$ is horizontally nondegenerate, $\Eval(g^{-1})$ defines a metric
on $V^*$.  Its dual is a metric on $V$, or equivalently a pseudo-Riemannian
metric on each fiber of $\pa X$.  We denote this induced metric on the
fibers by $g_F$.  

Concretely:  
\begin{equation}\label{ginverse}
g^{-1}=
\left(\begin{array}{ccc}
x\pa_x&x\pa_{y^\al}&\pa_{z^A}\\
\end{array}\right)
\left(\begin{array}{ccc}
\gb^{00}&\gb^{0\be}&\gb^{0B}\\
\gb^{\al 0}&\gb^{\al\be}&\gb^{\al B}\\
\gb^{A0}&\gb^{A\be}&\gb^{AB}
\end{array}\right)\left(\begin{array}{c}
x\pa_x\\
x\pa_{y^\be}\\
\pa_{z^B}\\
\end{array}\right),
\end{equation}
so
$\Eval(g^{-1})|_{\pa X}=\gb^{AB} \pa_{z^A}\pa_{z^B}$.  Nondegeneracy of   
$
\left(\begin{array}{cc}
\gb_{00}&\gb_{0\be}\\
\gb_{\al 0}&\gb_{\al\be}\\
\end{array}\right)
$
is equivalent to nondegeneracy of $\gb^{AB}$.
The induced metric on the fibers is $(g_F)_{AB}dz^Adz^B$, where
$(g_F)_{AB}=(\gb^{AB})^{-1}$.  This metric $g_F$ is independent of the
choice of all coordinates.  

Next we introduce the notion of an exact edge metric.  Let $g$ be an edge
metric and $x$ a defining function.  Now $\frac{dx}{x}$ is a smooth 
section of ${}^eT^*X$.  Let $X_x$ be the edge vector field dual to
$\frac{dx}{x}$ with respect to $g$.  $\Eval(X_x)|_{\pa X}$ is then a
section of $V$.  If $g$ is horizontally nondegenerate, we can 
define a section $\alpha_x$ of $V^*$ to be its dual with respect to $g_F$.
Thus we have associated to each defining function $x$ a 1-form $\alpha_x$
on the fibers of $\pa X$.  If $\widehat{x}=ax$ is another defining
function, then 
$\frac{d\widehat{x}}{\widehat{x}}=
\frac{dx}{x} + \frac{da}{a}=\frac{dx}{x} + d \log a$.  
Following through the definition shows that the corresponding 1-forms 
are related by $\al_{ax}=\al_x+d_V\log a$, where $d_V\log a=d\log a|_V$.
Thus 
$\al_x$ changes by an exact form under change of defining function.
\begin{definition}\label{exact}
An edge metric $g$ is said to be {\it exact} if $g$ is
horizontally nondegenerate and if for each defining function $x$, there is   
$f\in C^\infty(\pa X)$ so that $\al_x=d_Vf$.  
\end{definition}

\noindent
The above reasoning shows that if this holds for one $x$, it holds for all
$x$.  If $g$ is exact, then by correct choice of $a$ we can 
find $x$ so that $\al_x=0$.

\begin{definition}\label{grelated}
If $g$ is an exact edge metric, a defining function $x$ is said to be {\it
$g$-related} if $\al_x=0$.  
\end{definition}

\noindent
If $x$ is $g$-related, then another defining function $\widehat{x}=ax$ is
also $g$-related if and only if $d_Va=0$, i.e. $a|_{\pa X}$ is locally
constant on the fibers.  So $g$-related defining functions are determined
precisely up 
to multiplication by a positive function whose restriction to $\pa X$ is 
locally constant on the fibers.  

Clearly $x$ is $g$-related if and only if $\Eval(X_x)|_{\pa X}=0$.  
Since $g^{-1}$ is given by \eqref{ginverse},
we deduce that $x$ is $g$-related if and only if $\gb^{0B}|_{\pa X}=0$, and
this condition is independent of the choice of $y^\al$ and $z^A$.  

Recall that a horizontally nondegenerate edge metric $g$ induces a conformal
class of metrics on $TX|_{\pa X}/V$ with representatives $(x^2g)|_{\pa X}$.   
If $g$ is exact, we can restrict to representatives of the conformal class
which arise from $g$-related defining functions $x$.  We will call this the
{\it reduced conformal infinity} of $g$.  Metrics in the reduced conformal
infinity are determined up to rescaling by functions which are 
locally constant on the fibers of $\pa X$.  Choosing a representative
metric in the 
reduced conformal infinity is entirely equivalent to choosing a $g$-related
defining function $x$ to first order at $\pa X$.  

The following lemma will be useful in the sequel.
\begin{lemma}\label{normlem}
Let $g$ be an exact edge metric and $x$ a $g$-related defining function.  
Then $\big|\frac{dx}{x}\big|_H\big|^2_{g|_H} =|\frac{dx}{x}|^2_g$ on $\pa
X$.  
In particular, if $g$ is 
exact and normalized, then every $g$-related defining function is
$g$-normalized.  
\end{lemma}
\begin{proof}
Choose local coordinates $(x,y^\al,z^A)$, taking $x$ to be the given
$g$-related defining function.  Recall that if  
$$
M=
\left(\begin{array}{cc}
T & U\\
V & W \\
\end{array}\right)
$$
is an invertible matrix in block form with $T$ and $W$ square and $T$
invertible, then its inverse can be written
$$
M^{-1}=
\left(\begin{array}{cc}
T^{-1}+T^{-1}US^{-1}VT^{-1} & -T^{-1}US^{-1}\\
-S^{-1}VT^{-1} & S^{-1}\\ 
\end{array}\right), 
$$
where $S=W-VT^{-1}U$
necessarily is invertible.  Apply this with 
$$
M=
\left(\begin{array}{ccc}
\gb_{00}&\gb_{0\be}&\gb_{0B}\\
\gb_{\al 0}&\gb_{\al\be}&\gb_{\al B}\\
\gb_{A0}&\gb_{A\be}&\gb_{AB}
\end{array}\right),
$$
$$
T=
\left(\begin{array}{cc}
\gb_{00}&\gb_{0\be}\\
\gb_{\al 0}&\gb_{\al\be}\\
\end{array}\right),
\qquad
W=\left(\gb_{AB}\right),
\qquad
U=V^t=
\left(\begin{array}{c}
\gb_{0B}\\
\gb_{\al B}
\end{array}\right),
$$
all evaluated at $\pa X$.  
The hypothesis that $x$ is $g$-related says exactly that the first row 
of $T^{-1}US^{-1}$ vanishes.  Hence the first row of
$T^{-1}US^{-1}VT^{-1}$ also vanishes.  Hence the first row of 
$T^{-1}+T^{-1}US^{-1}VT^{-1}$ agrees with the first row of $T^{-1}$.  In
particular, their ${}^{00}$ components agree, which is the desired
conclusion.  
\end{proof}

\begin{remark}
A simpler proof can be given if one assumes that 
$|\frac{dx}{x}|^2_g\neq 0$ on $\pa X$.  This is of course automatic for $g$  
positive definite.  Under this hypothesis one can make a change of the 
$y$-variables 
$y^\al\rightarrow y^\al + \lambda^\al x$, with $\lambda^\al|_{\pa X}$
chosen to make $\gb^{0\al}|_{\pa X}=0$.  In the 
new coordinates one has $\gb^{0\al}=0$ and $\gb^{0A}=0$ on $\pa X$, and the
conclusion is clear.  
\end{remark}

Next we formulate the normal form condition.  If $X$ is an edge manifold,
then $\pa X$ is the total space of a fibration.  Consider 
$[0,\infty)\times \pa X$ as a manifold-with-boundary, with boundary 
$\{0\}\times \pa X\cong   \pa X$.  The given fibration of $\pa X$ induces a 
natural edge manifold structure on $[0,\infty)\times \pa X$.  The coordinate
$x$ of the first factor is a canonical defining function on 
$[0,\infty)\times \pa X$.  

\begin{definition}\label{normalformdef}
An edge metric $g$ on a neighborhood $\cU$ of $\{0\}\times \pa X$ in  
$[0,\infty)\times \pa X$ is in {\it normal form} if  
$
g=\frac{dx^2}{x^2} + k,
$
where $k$ is a smooth section of ${}^eT^*\cU$
satisfying $x\pa_x\into k =0$ everywhere.  
\end{definition}

This is equivalent to requiring that $g$ have the form
$$
g=\left(\begin{array}{ccc}
\frac{dx}{x}&
\frac{dy^{\bf \al}}{x}&
dz^A\\
\end{array}\right)
\left(\begin{array}{ccc}
1&0&0\\
0&\gb_{\al\be}&\gb_{\al B}\\
0&\gb_{A\be}&\gb_{AB}
\end{array}\right)\left(\begin{array}{c}
\frac{dx}{x}\\
\frac{dy^{\be}}{x}\\
dz^B\\
\end{array}\right),
$$
and $k$ is given by
$$
k=\left(\begin{array}{cc}
\frac{dy^{\bf \al}}{x}&
dz^A\\
\end{array}\right)
\left(\begin{array}{cc}
\gb_{\al\be}&\gb_{\al B}\\
\gb_{A\be}&\gb_{AB}
\end{array}\right)\left(\begin{array}{c}
\frac{dy^{\be}}{x}\\
dz^B\\
\end{array}\right).
$$
Observe that $g$ is exact and normalized.  Also $x$ is $g$-normalized and
$g$-related.  

The main theorem asserts that any exact, normalized edge metric $g$ can be 
put into normal form, and the normal forms for $g$ are parametrized
by the $g$-related defining functions to first order, or equivalently by
the representatives for the reduced conformal infinity.  

\begin{theorem}\label{main}
Let $X$ be an edge manifold and $g$ an exact, normalized edge metric.  If
$x_0$ is a $g$-related defining function, then there is a unique
diffeomorphism $\psi$ from a neighborhood of 
$\{0\}\times \pa X$ in $[0,\infty)\times \pa X$ to a neighborhood of $\pa
X$ in $X$, 
such that $\psi|_{\pa X}=Id$, $\psi^*g$ is in normal form, and
$\psi^*x_0=x+O(x^2)$.  
\end{theorem}

The main step in the proof is to solve the eikonal equation:
\begin{proposition}\label{eikonal}
Let $X$ be an edge manifold and $g$ an exact, normalized edge metric.  If
$x_0$ is a $g$-related defining function, then in a neighborhood of $\pa X$
there is a $g$-related 
defining function $\xh$, uniquely determined by the conditions 
$$
\Big| \frac{d\xh}{\xh}\Big|^2_g=1,\qquad\quad \xh =x_0+O(x_0^2).
$$
\end{proposition}

\medskip
\noindent
Theorem~\ref{main} follows from Proposition~\ref{eikonal} by the usual
argument of flowing along integral curves:  

\bigskip
\noindent
{\it Proof of Theorem~\ref{main}.}  
Let $\xh$ be as in Proposition~\ref{eikonal}.  
Recall that $X_{\xh}$ is the edge vector field dual to $d\xh/\xh$ with
respect to 
$g$, and $\Eval(X_{\xh})=0$ on $\pa X$ since $\xh$ is $g$-related.
Consequently  
$N:=\xh^{-1}\Eval(X_{\xh})$ is a smooth vector field up to $\pa X$, and 
$N\xh=|\frac{d\xh}{\xh}|^2_g=1$.  In particular, $N$ is transverse to $\pa
X$. For $x \geq 0$ and $p\in \pa X$, define $\psi(x,p)$ to be the 
point obtained by following the integral curve 
of $N$ emanating from $p$ for $x$ units of time.  Since $N\xh=1$, we
have $\psi^*\xh=x$, and $N$ is orthogonal to the level sets of
$\xh$ since $X_{\xh}$ is dual to $d\xh/\xh$.  Thus $\psi^*g$ has the desired 
form.  
\stopthm

We conclude this section by reducing Proposition~\ref{eikonal} to the
solution of  
a singular initial value problem of the form considered in
Theorem~\ref{pdetheorem}.  It suffices to 
prove Proposition~\ref{eikonal} 
locally in a neighborhood of a boundary point, since the uniqueness implies
that the local solutions will piece together to form a global solution.  
Relabel $x_0$ as $x$ and write $\xh=e^{\om}x$.  Our new unknown is 
$\om$, with boundary condition $\om=0$ at $x=0$.  Now 
$\frac{d\xh}{\xh} = \frac{dx}{x} +d\om$, so the equation 
$\big|\frac{d\xh}{\xh}\big|^2_g=1$ becomes 
\begin{equation}\label{expandeq}
2X_x\om +|d\om|^2_g = 1-\Big|\frac{dx}{x}\Big|^2_g,
\end{equation}
where we now neglect the distinction between $X_x$ and $\Eval(X_x)$.
Lemma~\ref{normlem} shows that $x$ is $g$-normalized, so the right-hand
side vanishes at $\pa X$.  Work in local coordinates
$(x,y^{\al},z^A)$ as above.  
The left-hand side is a quadratic polynomial
in $x\pa_x\om$, $x\pa_{y^\al}\om$, and $\pa_{z^A}\om$ with no constant term
and with coefficients smooth up to the boundary.  It follows that
\eqref{expandeq} can be written as
\begin{equation}\label{Q}
Q(x,y,z,x\pa_x\om,\pa_y\om,\pa_z\om)=f(x,y,z),
\end{equation}
where $Q$ is a quadratic 
polynomial in $(x\pa_x\om,\pa_y\om,\pa_z\om)$ with no constant term and
with coefficients depending on $(x,y,z)$ which are smooth up to $x=0$, and 
$f$ is smooth with $f(0,y,z)=0$.  (We have absorbed the $x$ multiplying
$\pa_y\om$ into the coefficients.)  We make the following observations
about $Q$.  First, the coefficient of the linear term $x\pa_x\om$ is
nonzero at $x=0$, since 
$X_xx=x\big|\frac{dx}{x}\big|^2_g$.  Second, the coefficients of the linear
terms $\pa_y\om$ and $\pa_z\om$ vanish at $x=0$, since $x$ is $g$-related
so that $X_x=0$ at $x=0$.  Third, all of the arguments
$(x\pa_x\om,\pa_y\om,\pa_z\om)$ themselves vanish at $x=0$ when evaluated
on any function $\om$ satisfying the initial condition $\om=0$ at $x=0$.
In particular, the partial derivative of the quadratic terms of $Q$ with
respect to any of $x\pa_x\om$, $\pa_y\om$, $\pa_z\om$ vanishes at $x=0$
when evaluated on the initial data. 

The implicit function theorem (or the quadratic formula) implies that in a
neighborhood of $(x,y,z,x\pa_x\om,\pa_y\om,\pa_z\om)=(0,y,z,0,0,0)$,
\eqref{Q} may be solved for $x\pa_x\om$.  So it may be written in the form  
\begin{equation}\label{reduced}
x\pa_x\om = F(x,y,z,\pa_y\om,\pa_z\om),
\end{equation}
where $F$ is a smooth function of its arguments satisfying
$F(0,y,z,0,0)=0$.  
Moreover, the observations above show that  
$F_{\pa_y\om}(0,y,z,0,0)=0$ and $F_{\pa_z\om}(0,y,z,0,0)=0$.
Equation \eqref{reduced} with initial condition $\om=0$ at $x=0$ is of the 
form \eqref{ivp}, where $y$ in \eqref{ivp} plays the role of $(y,z)$ in
\eqref{reduced}.  Condition \eqref{second} holds since $F$ in  
\eqref{reduced} is independent of $\om$, so that $F_\om\equiv 0$.  Thus
Proposition~\ref{eikonal} follows from Theorem~\ref{pdetheorem}.   

\section{Singular Initial Value Problems}\label{analysis}

In this section we prove Theorem~\ref{pdetheorem}.  First observe that 
the conclusion in Theorem~\ref{pdetheorem} fails without the hypothesis   
$F_\om(0,y,\om_0(y),\pa_y\om_0(y))<1$.  For instance, the equation
$x\pa_x\om = \om$ has infinitely many smooth solutions $\om = cx$
satisfying $\om(0)=0$, and the equation $x\pa_x\om = \om +x$ has no smooth 
solutions (the general solution is $\om = cx+x\log x$).  Also note  
that if 
$0<F_\om(0,y,\om_0(y),\pa_y\om_0(y))<1$, then the smooth solution need not
be the only continuous solution. For example, if $0<\al<1$, then $\om =cx^\al$ 
solves $x\pa_x\om = \alpha \om$ with $\om(0)=0$ for any $c\in \R$.  In this
case the unique smooth solution is $\om = 0$.     

We first use a standard reduction technique via Taylor
expansion to reduce the equation to a simpler form.  In the following we
denote 
$F_\om^{(0)}(y)=F_\om(0,y,\om_0(y),\pa_y\om_0(y))$ and similarly for
other derivatives of $F$ evaluated on the initial data.

Observe first that differentiating \eqref{ivp} with respect to $x$ at 
$x=0$ and solving for $\om_x$ shows that if $\om$ is a smooth solution,
then 
$$
\om_x(0,y)
=\frac{F_x^{(0)}(y)}{1-F_\om^{(0)}(y)}:=\om_1(y).      
$$
We can write 
$$
\om(x,y)=\om_0(y)+x\big(\om_1(y)+u(x,y)\big) 
$$
for a smooth function $u(x,y)$, and regard $u$ as the new unknown.  
\begin{proposition}\label{ivpprop}
In terms of $u$, \eqref{ivp} becomes 
\begin{equation}\label{ivpu}
x\pa_xu= (F_\om^{(0)}(y)-1)u+xG(x,y,u,\pa_yu),\qquad u(0,y)=0,
\end{equation}
where $G$ is a smooth function of $(x,y,u,q)$.  
\end{proposition}
\begin{proof}
It is clear from the discussion above that the initial condition on $u$ is 
$u(0,y)=0$.    
Set $q_0(y)=\pa_y\om_0(y)$.  The second order Taylor expansion of
$F(x,y,\om,q)$ about $(0,y,\om_0(y),q_0(y))$ takes the form 
$$
F(x,y,\om,q)=F_x^{(0)}(y)x+F_\om^{(0)}(y)(\om-\om_0(y))
+Q(x,\om-\om_0(y),q-q_0(y)),
$$
where $Q$ is a homogeneous quadratic polynomial of its arguments with
coefficients which are smooth functions of $(x,y,\om,q)$.  We have 
$$
\pa_x\om = \om_1 +(x\pa_x+1)u,\qquad
\pa_y\om-q_0(y)=x(\pa_y\om_1+\pa_yu).  
$$
Substituting and then dividing by $x$ shows that \eqref{ivp} becomes  
$$
\om_1 +(x\pa_x+1)u = F_x^{(0)}+F_\om^{(0)}(\om_1+u) + xG(x,y,u,\pa_yu)
$$
for a smooth function $G$.  The definition of $\om_1$ shows that
$\om_1=F_x^{(0)} + F_\om^{(0)}\om_1$, so this reduces to \eqref{ivpu}.  
\end{proof}

Proposition~\ref{ivpprop} implies that Theorem~\ref{pdetheorem} follows
from the following special case.  

\begin{proposition}\label{pdeprop}
Let $b(y)$ and $G(x,y,u,q)$ be smooth and suppose $b(y)<0$.  Then the IVP 
\begin{equation}\label{ivpreduced}
x\pa_xu = b(y)u +xG(x,y,u,\pa_yu),\qquad u(0,y)=0
\end{equation}
has a unique smooth solution for sufficiently small $x\geq 0$.  
\end{proposition}

We prove Proposition~\ref{pdeprop} by an adaptation of the method of
characteristics.  The main
tool is a result asserting the existence and uniqueness of smooth
``characteristic integral curves'' of time-dependent vector fields
vanishing at an initial point.    

Let $M$ be a smooth manifold and $p_0\in M$.  Suppose that $V(t,p)$ is a 
smooth time-dependent vector field defined for $t$ near $0$ and $p$ in a 
neighborhood of $p_0$, such that 
$V(0,p_0)=0$.  By a characteristic integral curve for $V$ at $p_0$ we mean
a curve 
$\ga:[0,\ep)\rightarrow M$ for some $\ep>0$ such that 
\begin{equation}\label{ivpode}
t\frac{d}{dt}\ga (t)=V(t,\ga(t)),\quad \ga(0)=p_0.
\end{equation}
Recall that the linearization of a vector
field at a zero is the endomorphism $DV$ of $T_{p_0}M$ such that 
$V(p)=DV(p-p_0)$ to first order at $p_0$.  For a time-dependent vector
field this refers to the linearization of the vector field in the space
variables with $t$ fixed.  
\begin{theorem}\label{vectfield}
Let $V(t,p)$ be a smooth time-dependent vector field such that
$V(0,p_0)=0$.  Suppose that all eigenvalues $\lambda$ of $DV(0,p_0)$
satisfy $\Re \lambda <1$.  Then on a sufficiently small time interval there
exists a unique smooth characteristic 
integral curve for $V$ at $p_0$.  This characteristic integral curve
depends smoothly on variations of the initial point $p_0$ for which the
conditions $V(0,p_0)=0$ and $\Re\lambda<1$ continue to hold.    
\end{theorem}

Observe that the case $M=\R$ of Theorem~\ref{vectfield} coincides precisely
with the special case 
$n=1$ of Theorem~\ref{pdetheorem}, upon relabeling $t$
as $x$, $\gamma$ as $\omega$, and $V$ as $F$.  In particular, the examples
above show that  
existence and uniqueness of smooth solutions can fail if $\lambda =1$, and
there may be continuous solutions which are not smooth if 
$0<\lambda<1$.     

The first step in the proof of Theorem~\ref{vectfield} is to perform a
Taylor expansion 
analogous to the one made above for the pde.  Work in local coordinates 
on $M$ and let $x_0$ denote the coordinates of $p_0$.  Differentiating  
\eqref{ivpode} with respect to $t$ at $t=0$ and solving for $\gamma'(0)$
gives
$$
\ga'(0)=[I-DV(0,x_0)]^{-1}V_t(0,x_0):=\ga_1.
$$
Write
\begin{equation}\label{sigma}
\ga(t)=x_0 +t\big(\ga_1 + \sigma(t)\big). 
\end{equation}
Upon Taylor expanding $V(t,x)$ about $(0,x_0)$, substituting \eqref{sigma},
and simplfying as in the proof of Proposition~\ref{ivpprop}, one finds  
that when written in terms of $\sigma$, \eqref{ivpode} takes the form 
\begin{equation}\label{ivpodes}
t\sigma' + A\sigma  = tG(t,\sigma), \quad \sigma(0)=0.
\end{equation}
Here $A=I-DV(0,x_0)$ has the property that all of its eigenvalues have
positive real part.  $G$ is smooth, and $A$ and $G$ depend smoothly on the 
initial point $x_0$.  Initial value problems of the form \eqref{ivpodes}
are 
studied in Chapter 5 of \cite{Ki}.  The results formulated there assume that 
$A$ is independent of the parameters, but the same arguments apply to our 
situation.  We briefly outline a
proof that \eqref{ivpodes} has a unique smooth solution varying smoothly
with the parameters $x_0$ if the eigenvalues of $A$ have positive real
part.  Theorem~\ref{vectfield} is then a consequence by the reduction
above. 

The problem \eqref{ivpodes} can be reformulated as the integral equation 
\begin{equation}\label{inteq}
\sigma(t)=(T\sigma)(t):=t\int_0^1s^AG(st,\sigma(st))ds.
\end{equation}
The hypothesis that the eigenvalues of $A$ have positive real part implies
that the operators $s^A$ are uniformly bounded for $0<s\leq 1$.  
A standard contraction mapping/fixed point argument proves the existence
and uniqueness of a continuous solution.  
To establish smoothness in $t$, rewrite \eqref{inteq} as
$$
\sigma(t)= t^{-A}\int_0^ts^AG(s,\sigma(s))ds.
$$  
This shows that $\sigma$ is $C^1$ for $t>0$.  Differentiate in 
$t$ and change variables back to obtain  
\begin{equation}\label{deriv}
\sigma'(t)= G(t,\sigma(t))-A\int_0^1s^AG(st,\sigma(st))ds.
\end{equation}
Thus $\sigma$ is $C^1$ up to $t=0$.  Now 
successively differentiating \eqref{deriv} shows that $\sigma$ is
$C^{\infty}$.  Smoothness of $\sigma$ with respect to the parameters is a
consequence of the implicit function theorem applied to the equation 
$\sigma -T\sigma=0$.  

\bigskip
\noindent
{\it Proof of Proposition~\ref{pdeprop}.}
We construct a singular version of a Hamiltonian
flow-out in the first jet bundle of the solution $u$.  The argument follows 
the usual reasoning for the non-characteristic case, substituting 
Theorem~\ref{vectfield} in an appropriate parameterization for the 
existence and uniqueness of integral curves 
of the Hamiltonian vector field.  

Let $\cJ$ denote the first jet bundle of a scalar function $u$ on $\R^n$,
with coordinates $(x,y,u,p,q)$ where $p$ is the variable dual to $x$, and 
projection $\pi:\cJ\rightarrow \R^{n}$ given by $\pi(x,y,u,p,q)=(x,y)$.  
Set $\bx=(x,y)$ and $\bp=(p,q)$.  The 1-jet of a function $u$ on $\R^n$ is
the 
section of $\cJ$ given by $\bx\rightarrow (\bx,u(\bx),du(\bx))$.  We
denote its image $\{(\bx,u(\bx),du(\bx))\}$ by $\cS_u$; this is a
submanifold of $\cJ$ of dimension $n$.  
The tautological contact form is $\theta=du-\bp_id\bx^i$.  $\cS_u$ is a
Legendrian submanfold relative to $\theta$; i.e. the pullback of $\theta$
to $\cS_u$ vanishes.  

Recall that if  
$H(\bx,u,\bp)$ is a smooth real function on $\cJ$, the associated
Hamiltonian vector field is 
$$
\xi_H=H_{\bp_i}\pa_{\bx^i} +
\bp_iH_{\bp_i}\pa_u-(H_{\bx^i}+\bp_iH_u)\pa_{\bp_i}.
$$
It is uniquely determined by the conditions
\begin{equation}\label{Hprop1}
\xi_H\into d\theta = dH\mod \theta\qquad\qquad \theta(\xi_H)=0
\end{equation}
and satisfies
\begin{equation}\label{Hprop2}
\xi_H\into d\theta = dH -H_u\theta,\qquad \xi_HH=0.
\end{equation}
If $u$ is a solution of $H(\bx,u,du)=0$, then $\xi_H$ is tangent to  
$\cS_u$ at all points of $\cS_u$.  

Take $H$ to be the Hamiltonian corresponding to \eqref{ivpreduced},
i.e. 
$$
H(x,y,u,p,q)=xp-b(y)u-xG(x,y,u,q).
$$
Differentiating \eqref{ivpreduced} at $x=0$ shows that a solution $u$ must
satisfy 
\begin{equation}\label{initderiv}
\pa_xu(0,y)=\frac{G(0,y,0,0)}{1-b(y)}:= p_0(y).
\end{equation}
Define a smooth submanifold $\cI\subset\cJ$ of dimension $n-1$ by 
$$
\cI=\{(0,y,0,p_0(y),0)\}.
$$
Hamilton's equations for the integral curves of $\xi_H$ take
the form
\[
\begin{split}
\frac{dx}{ds}&=x\\
\frac{dy}{ds}&=-xG_q\\
\frac{du}{ds}&=x(p-q_iG_{q_i})\\
\frac{dp}{ds}&=p(b(y)-1)+G+x(G_x+pG_u)\\
\frac{dq}{ds}&=b_y(y)u+b(y)q+x(G_y+qG_u).
\end{split}
\]
Observe that $\xi_H$ vanishes identically on $\cI$.  So all integral curves
of $\xi_H$ beginning on $\cI$ are constant; there is no Hamiltonian
flow-out in the usual sense.  Instead we consider characteristic integral
curves of $\xi_H$ beginning on $\cI$.  $\xi_H$ is time-independent and
constants are also characteristic 
integral curves.  But the characteristic integral curves  
are not unique:  $D\xi_H$ on $\cI$ has $\lambda =1$ as an eigenvalue
arising from the first equation in the system above.  By using $x$
as the parameter, we will obtain unique nonconstant  
characteristic integral curves of $\xi_H$ emanating from $\cI$ whose union 
will form the submanifold $\cS_u$ giving the solution $u$. 

Use $x$ as a parameter for the characteristic integral curves.  The
first equation above gives $d/ds=xd/dx$.  Substituting in the remaining
equations gives the system
\begin{equation}\label{xtime}
\begin{split}
x\frac{dy}{dx}&=-xG_q\\
x\frac{du}{dx}&=x(p-q_iG_{q_i})\\
x\frac{dp}{dx}&=p(b(y)-1)+G+x(G_x+pG_u)\\
x\frac{dq}{dx}&=b_y(y)u+b(y)q+x(G_y+qG_u).  
\end{split}
\end{equation}
This has the form \eqref{ivpode}, where $x$ plays the role of $t$.  
Choose $y_0\in \R^{n-1}$ and impose initial conditions 
\begin{equation}\label{init}
y(0)=y_0,\quad u(0)=0,\quad p(0)=p_0(y_0),\quad q(0)=0.
\end{equation}
The linearization of the right-hand side of \eqref{xtime} evaluated at
$x=0$ and at the given initial conditions for $(y,u,p,q)$ is 
$$D=
\begin{pmatrix}
0&0&0&0\\
0&0&0&0\\
*&*&b(y_0)-1&*\\
*&*&0&b(y_0)I
\end{pmatrix},
$$
where the blocks have sizes $n-1,1,1,n-1$.  Here $*$ denotes a  
quantity whose value will be irrelevant and $I$ denotes the
$(n-1)\times (n-1)$ identity matrix.  The eigenvalues of $D$ are  
$0$ with multiplicity $n$, $b(y_0)-1$ with multiplicity $1$, and $b(y_0)$ 
with multiplicity $n-1$.  These are all real and less than $1$, so 
Theorem~\ref{vectfield} implies that there is a unique smooth solution  
$(y(x,y_0),u(x,y_0),p(x,y_0),q(x,y_0))$ of \eqref{xtime}, \eqref{init}
for sufficiently small $x$ varying smoothly with $y_0$. 

Define a map $\Phi$ into $\cJ$ by
$$
\Phi(x,y_0)=(x,y(x,y_0),u(x,y_0),p(x,y_0),q(x,y_0)).
$$
Since $y(0,y_0)=y_0$, it follows that $\Phi$ is
a diffeomorphism from a neighborhood of $\{x=0\}$ to a submanifold
$\cF\subset \cJ$ of dimension $n$.  We claim that $H=0$ on $\cF$.  
Since $\xi_Hx=x$ and $y_0$ is constant on the 
solution curves, it follows that $\Phi^*\xi_H=x\pa_x$.  Since $\xi_HH=0$,
we have
$x\pa_x(\Phi^*H)=0$.  Since $H=0$ on $\cI$, one concludes that $H=0$ on 
$\cF$ as claimed.  

We now prove existence in Proposition~\ref{pdeprop}.  
The projection $\pi:\cJ\rightarrow \R^{n}$ 
restricts to a diffeomorphism from $\cF$ to a neighborhood of $\{x=0\}$
(possibly after shrinking $\cF$).  Therefore on $\cF$ we can regard $u,p,q$
as functions of $(x,y)$.  In particular this defines a smooth function
$u(x,y)$.  We claim that  
\begin{equation}\label{newpq}
p(x,y)=\pa_xu(x,y),\qquad q_i(x,y)=\pa_{y^i}u(x,y).
\end{equation}
This is equivalent to saying that $\cF=\cS_u$.    
Existence in Proposition~\ref{pdeprop} follows immediately, 
as then the equation $H=0$ on $\cF$
together with the initial condition become the statement that $u$ satisfies   
\eqref{ivpreduced}.  

Since $\theta=du-pdx-q_idy^i$, in order to prove \eqref{newpq} it suffices 
to show that the pullback of $\theta$ to $\cF$ vanishes.  
Recalling \eqref{Hprop1}, \eqref{Hprop2}, observe that 
$$
\cL_{\xi_H}\theta = \xi_H\into d\theta + d(\theta(\xi_H))
= dH - H_u\theta,
$$
where $\cL$ denotes the Lie derivative.  
For the pullback to $\cF$ we therefore obtain
$\cL_{\xi_H}\theta =-H_u\theta$.  Pulling back by $\Phi$ gives
$\cL_{x\pa_x}\Phi^*\theta = -\left(\Phi^*H_u \right)\Phi^*\theta$.  If we
write  
$\Phi^*\theta=\theta_0 dx + \theta_id(y_0)^i$, then this becomes
\begin{equation}\label{thetai}
x\pa_x\theta_0 = \left(-\Phi^*H_u-1\right) \theta_0,\qquad
x\pa_x\theta_i = \left(-\Phi^*H_u\right) \theta_i.
\end{equation}
For each $y_0$, these are scalar ode's of the form 
$x\pa_xv=\beta(x)v$, where $v=\theta_0$ or $\theta_i$ and 
$\beta=-\Phi^*H_u-1$ or $-\Phi^*H_u$.  Since $H_u=-b(y)$ at $x=0$ and
$b<0$, we have 
$\beta<0$ near $x=0$.  The general solution is 
$$
v=c\exp{\int^x\frac{\beta(s)\, ds}{s}}=cx^{\beta(0)}h(x),
$$
where $c\in \R$ and $h$ is a nonvanishing smooth function.  Since 
$\beta(0)<0$ and $\theta$ is smooth, we must have $c=0$, so we obtain
$\theta_0=\theta_i=0$ as 
desired.  (Alternately, the vanishing of $\theta_0$ and $\theta_i$ 
follows from uniqueness in Theorem~\ref{vectfield} applied to  
\eqref{thetai}.)

Finally we prove uniqueness in Proposition~\ref{pdeprop}.  
We show that if $u$ is any smooth solution of \eqref{ivpreduced}, then
$\cS_u=\cF$ near $\cI$.  We have already observed that  
$\pa_xu(0,y)$ is given by \eqref{initderiv}, so that 
$\cS_u\cap \{x=0\}=\cI=\cF\cap \{x=0\}$.  The system 
$$
x\frac{dy}{dx}=-xG_q(x,y,u(x,y),\pa_yu(x,y)),\qquad y(0)=y_0
$$
for unknown $y(x)$ has a unique smooth solution (either by
Theorem~\ref{vectfield} or by cancelling $x$ and quoting usual ode
theory).  As $y_0$ varies, the corresponding curves $(x,y(x))$ fill out a
neighborhood of $\{x=0\}$ in $\R^n\cap \{x\geq 0\}$.  Therefore near $\cI$,
$\cS_u$ is the union over $y_0$ of the lifts    
\begin{equation}\label{lifts}
x\mapsto (x,y(x),u(x,y(x)),\pa_xu(x,y(x)),\pa_yu(x,y(x))).
\end{equation}
The curves 
$$
x\mapsto (y(x),u(x,y(x)),\pa_xu(x,y(x)),\pa_yu(x,y(x)))
$$
solve \eqref{xtime} since $\xi_H$ is everywhere tangent to $\cS_u$.   
Since $\cF$ was defined to be the union of  
all curves \eqref{lifts} corresponding to solutions of \eqref{xtime},
it follows that $\cS_u=\cF$ near $\cI$.      
\stopthm

\end{document}